\documentclass[12pt]{article}

\usepackage{amsmath, amssymb, amsthm}
\usepackage{graphicx}
\graphicspath{%
    {converted_graphics/}
}

\newtheorem{Theorem}{Theorem}
\newtheorem{Corollary}{Corollary}
\newtheorem{Example}{Example}
\newtheorem{Lemma}{Lemma}
\newtheorem{Conjecture}{Conjecture}

\begin{document}

\title{Wilf Equivalence for the Charge Statistic }        
\author{Kendra Killpatrick}        
\date{\today}         
\maketitle

\begin{abstract}
Savage and Sagan have recently defined a notion of {\it {st}}-Wilf equivalence for any permutation statistic {\it {st}} and any two sets of permutations $\Pi$ and $\Pi'$.  In this paper we give a thorough investigation of {\it {st}}-Wilf equivalence for the charge statistic on permutations and use a bijection between the charge statistic and the major index to prove a conjecture of Dokos, Dwyer, Johnson, Sagan and Selsor regarding powers of 2 and the major index.  
\end{abstract}

\section{Background}      

Let $S_n$ be the symmetric group of all permutations of the set $[n] = \{ 1, 2, \dots, n\}$ and suppose $\pi = a_1 a_2 \cdots a_n$ and $\sigma = b_1 b_2 \cdots b_n$ are two permutations in $S_n$.  We say that $\pi$ is {\it {order isomorphic}} to $\sigma$ if $a_i < a_j$ if and only if $b_i < b_j$.  For any $\pi \in S_n$ and $\sigma \in S_k$ for $k \leq n$, we say that $\pi$ {\it {contains a copy of $\sigma$}} if $\pi$ has a subsequence that is order isomorphic to $\sigma$.  If $\pi$ contains no subsequence order isomorphic to $\sigma$ then we say that $\pi$ {\it {avoids}} $\sigma$.  

Now let $\Pi$ be a subset of permutations in $S_n$ and define $Av_n(\Pi)$ as the set of permutations in $S_n$ which avoid every permutation in $\Pi$.  Two sets of permutations $\Pi$ and $\Pi'$ are said to be {\it {Wilf equivalent}} if $\vert Av_n(\Pi)\vert = \vert Av_n(\Pi')\vert$.  If $\Pi$ and $\Pi'$ are Wilf equivalent, we write $\Pi \equiv \Pi'$.  

Savage and Sagan \cite{SSa} defined a $q$-analogue of Wilf equivalence by considering any permutation statistic $st$ from $S_n \rightarrow N$, where $N$ is the set of nonnegative integers, and letting 
\[
F_n^{st}(\Pi;q) = \sum_{\sigma \in Av_n(\Pi)} q^{st(\sigma)}.
\]
They defined $\Pi$ and $\Pi'$ to be {\it {st-Wilf equivalent}} if $F_n^{st}(\Pi;q) = F_n^{st}(\Pi';q)$ for all $n \geq 0$.  In this case, we write $\Pi \stackrel{st}{\equiv} \Pi'$. We will use $[\Pi]_{st}$ to denote the st-Wilf equivalence class of $\Pi$.  If we set $q=1$ in the generating function above we have $F_n^{st}(\Pi;1) = \vert Av_n(\Pi) \vert$, thus st-Wilf equivalence implies Wilf equivalence.  

In \cite{DDJ}, Dokos, Dwyer, Johnson, Sagan and Selsor give a thorough investigation of st-Wilf equivalence for both the major index and the inversion statistic.  Our goal in this paper is to give a similarly thorough investigation for another well known Mahonian statistic, the charge statistic.  In Section 2, we give the necessary definitions for the material covered and in Section 3 we discuss charge Wilf equivalence for subsets $\Pi \in S_3$.  In Section 4 we state and prove a conjecture of Dokos, Dwyer, Johnson, Sagan and Selsor by showing it is equivalent to a similar statement for the charge statistic.  Our proof of this result uses some basic facts about standard Young tableaux and the Robinson-Schensted correspondence.  We close the paper with some directions for further research.

\section{Definitions}

Throughout this paper we will utilize some basic operations on permutations, namely the {\it{inverse}}, the {\it {reverse}} and the {\it{complement}}.  For a permutation $\pi = \pi_1 \pi_2 \cdots \pi_n$, the inverse is the standard inverse operation on permutations, the reverse is
\[
\begin{array}{cccccc}
\pi^r &=& \pi_n& \cdots &\pi_2 &\pi_1
\end{array}
\]
and the complement is 
\[
\begin{array}{cccccc}
\pi^c &=& n+1-\pi_1&n+1-\pi_2&\cdots&n+1-\pi_n.
\end{array}
\]

For a permutation $\pi  = \pi_1 \pi_2 \cdots \pi_n \in S_n$, define the {\it {descent set of $\pi$}} to be $Des(\pi) = \{ i \vert \pi_i > \pi_{i+1} \}$.  The {\it {major index}} of a permutation, first defined by MacMahon \cite{Mac}, is then defined as 
\[
maj(\pi) = \sum_{i \in Des(\pi)} i.
\]

For example, for $\pi = \begin{array}{ccccccccc} 3&2&8&5&7&4&6&1&9 \end{array}$, $Des(\pi) = \{1, 3, 5, 7\}$ and $maj(\pi) = 1 + 3 + 5 + 7 = 16$.  

Let $\pi$ be a permutation in $S_n$.  For any $i$ in the permutation, define the {\it {charge value of $i$}}, $chv(i)$, recursively as follows:

\begin{align*}
chv(1) &= 0 \\
chv(i) &= 0 \textnormal{ if $i$ is to the right of $i-1$ in $\pi$} \\
chv(i) &= n+1-i \textnormal{ if $i$ is to the left of $i-1$ in $\pi$}
\end{align*}

Now for $\pi \in S_n$, define the {\it{charge of $\pi$}}, $ch(\pi)$, to be
\[
ch(\pi) = \sum_{i=1}^n chv(i).
\]

In the following example for $\pi = \begin{array}{ccccccccc} 3&2&8&5&7&4&6&1&9 \end{array}$, the charge values of each element are given below the permutation:
\[
\begin{array}{ccccccccccc}
\pi&=&3&2&8&5&7&4&6&1&9\\
 & &\textnormal{\scriptsize7}&\textnormal{\scriptsize8}&\textnormal{\scriptsize2}&\textnormal{\scriptsize5}&\textnormal{\scriptsize3}&\textnormal{\scriptsize0}&\textnormal{\scriptsize0}&\textnormal{\scriptsize0}&\textnormal{\scriptsize0}
\end{array}
\]
end $ch(\pi) = 7+8+2+5+3=25$.  The definition of the charge statistic was first given by Lascoux and Sch\"utzenberger \cite{LSc}.





Define 
\[
Ch_n(\Pi;q) = F_n^{ch}(\Pi; q) = \sum_{\sigma \in Av_n(\Pi)} q^{ch(\sigma)}.
\]

\section{Equivalence for permutations in $S_3$}

In this section, we will consider the polynomials $Ch_n(\Pi;q)$ where $\Pi \subseteq S_3$.  To begin, fix $n \geq 0$ and let $\pi \in S_n$.  Define $f(\pi) = ((\pi^r)^c)^{-1}$.  It is well known that each of the operations of reverse, complement and inverse are bijections on $S_n$ so $f$ is a bijection from $S_n$ to $S_n$.  

\begin{Lemma}
Fix $n \geq 0$ and let $\pi \in S_n$.  Then $maj(\pi) = ch(f(\pi))$.
\end{Lemma}

\begin{proof}

We will show that if there is no descent in position $i$ in $\pi$ then $n+1-i$ has a charge value of 0 in $f(\pi)$ and if there is a descent in position $i$ in $\pi$ then the charge value of $n+1-i$ in $f(\pi)$ is $(n+1) - (n+1-i) = i$.  

If there is no descent in position $i$ in $\pi$, then $\pi_i < \pi_{i+1}$.  Let $\pi_i = j$ and $\pi_{i+1} = k$, so $j<k$.  When we apply the reverse operation to $\pi$, we obtain $\pi_{n-i} = k$ and $\pi_{n+1-i} = j$.  Then $(\pi_{n-i})^c = n+1-k$ and $(\pi_{n-i+1})^c = n+1-j$ and since $j<k$, $n+1-k<n+1-j$.  When we apply the inverse operation we have $\pi_{n+1-k}=n-i$ and $\pi_{n+1-j}=n+1-i$, thus $n+1-i$ is to the right of $n-i$ in $f(\pi)$ so $n+1-i$ has a charge value of 0.

If there is a descent in position $i$ in $\pi$ then $\pi_i > \pi_{i+1}$.  Let $\pi_i=j$ and $\pi_{i+1}=k$, so $j>k$.  When we apply the reverse operation, we obtain $\pi_{n-i}=k$ and $\pi_{n+1-i}=j$.  Then $(\pi_{n-i})^c = n+1-k$ and $(\pi_{n+1-i})^c = n+1-j$ and since $j>k$, $n+1-k>n+1-j$.  When we apply the inverse operation, we have $\pi_{n+1-k} = n-i$ and $\pi_{n+1-j} = n+1-i$, thus $n+1-i$ is to the left of $n-i$ in $f(\pi)$ so $n+1-i$ has a charge value of $(n+1)=(n+1-i) = i$.
\end{proof}

\begin{Lemma} Fix $n \geq 0$.  Then
\begin{align*}
f:& Av_n(123) \rightarrow Av_n(123)\\
f:& Av_n(132) \rightarrow Av_n(213)\\
f:& Av_n(213) \rightarrow Av_n(132)\\
f:& Av_n(231) \rightarrow Av_n(231)\\
f:& Av_n(312) \rightarrow Av_n(312)\\
f:& Av_n(321) \rightarrow Av_n(321).
\end{align*}
\end{Lemma}

\begin{proof}
We will instead prove that for a fixed $n \geq 0$,
\begin{align*}
f:& Av_n(123)^C \rightarrow Av_n(123)^C\\
f:& Av_n(132)^C \rightarrow Av_n(213)^C\\
f:& Av_n(213)^C \rightarrow Av_n(132)^C\\
f:& Av_n(231)^C \rightarrow Av_n(231)^C\\
f:& Av_n(312)^C \rightarrow Av_n(312)^C\\
f:& Av_n(321)^C \rightarrow Av_n(321)^C.
\end{align*}

Suppose $\pi \in S_n$ contains a 123-pattern.  Then there exists an $i < j < k$ such that $\pi_i = a$, $\pi_j = b$ and $\pi_k = c$ with $a<b<c$.  Then in $\pi^r$, $\pi_{n+1-k} = c$, $\pi_{n+1-j}=b$ and $\pi_{n+1-i} = a$.  In $(\pi^r)^c$ we have $\pi_{n+1-k}=n+1-c$, $\pi_{n+1-j} = n+1-b$ and $\pi_{n+1-i}=n+1-a$ where $n+1-c < n+1-b < n+1-a$.  Finally, in $((\pi^r)^c)^{-1}$ we have $\pi_{n+1-c}=n+1-k$, $\pi_{n+1-b} = n+1-j$ and $\pi_{n+1-a} = n+1-i$.  Since $i < j < k$, $n+1-k < n+1-j < n+1-k$ thus $\pi_{n+1-c}$, $\pi_{n+1-b}$ and $\pi_{n+1-a}$ form a $(123)$-pattern in $f(\pi)$.
The proofs of the other bijections are similar and are left to the reader.

\end{proof}

We now prove the following result:

\begin{Theorem}
We have
\begin{align*}
[123]_{ch} &= \{ 123 \}\\
[321]_{ch} &= \{ 321 \} \\
[312]_{ch} &= \{ 312, 132 \} = [132]_{ch}\\
[213]_{ch} &= \{213, 231 \} = [231]_{ch}.
\end{align*}
\end{Theorem}

\begin{proof}

Dokos, Dwyer, Johnson, Sagan and Selsor \cite{DDJ} proved 

\begin{align*}
[123]_{maj} &= \{ 123 \}\\
[321]_{maj} &= \{ 321 \} \\
[132]_{maj} &= \{ 132, 231 \} = [231]_{maj}\\
[213]_{maj} &= \{213, 312 \} = [312]_{maj}.
\end{align*}

Recall that 
\[
[123]_{maj} = \{ \pi \in S_3 \vert F_n^{maj}(123; q) = F_n^{maj} (\pi; q) \} \]
or
\[ [123]_{maj} = \{ \pi \in S_3 \vert \sum_{\sigma \in Av_n(123)} q^{maj(\sigma)} = \sum_{\sigma \in Av_n(\pi)} q^{maj(\sigma)} \}.
\]

Since the function $f$ defined above takes $Av_n(123)$ to $Av_n(123)$ and takes the major index to the charge statistic, we can apply the function $f$ to the equation above to obtian $[123]_{ch} = \{ 123 \}$.  Similarly, by Lemma 2 we have that $f$ takes $Av_n(321)$ to $Av_n(321)$, thus $[321]_{ch} = \{ 321 \}$.  Since $f$ takes $Av_n(213)$ to $Av_n(132)$ and $Av_n(312)$ to $Av_n(312)$ we have $[132]_{ch} = [312]_{ch} = \{ 312, 132 \}$.  Finally, since $f$ takes $Av_n(231)$ to $Av_n(231)$ and $Av_n(132)$ to $Av_n(213)$ we have $[213]_{ch} = [231]_{ch} = \{ 213, 231 \}$.  

\end{proof}

Utilizing the same function $f$ and results of \cite{DDJ}, we can obtain the following results for larger subsets of $S_3$.

\begin{Theorem}
For $\Pi \in S_3$ with $\vert \Pi \vert = 2$ and $\Pi \neq \{123, 321\}$, we have
\[
[132, 213]_{ch} = \{ \{132, 213\}, \{213, 312\}, \{132, 231\}, \{231, 312\} \}.
\]

All other $ch$-Wilf equivalence classes under the given conditions contain a single pair.
\end{Theorem}

\begin{proof}
Dokos, Dwyer, Johnson, Sagan and Selsor prove the following theorem (Theorem 5.1) in \cite{DDJ}: 
\begin{Theorem}
We have
\[
[132, 213]_{maj} = \{ \{132, 213\}, \{132, 312\}, \{213, 231\}, \{231, 312\} \}.
\]
All other $maj$-Wilf equivalence classes for $\Pi \in S_3$ with $\vert \Pi \vert = 2$ and $\Pi \neq \{123, 321\}$ contain a single pair.
\end{Theorem}

Applying our function $f$ to this result gives our theorem for the charge statistic.
\end{proof}

Dokos, Dwyer, Johnson, Sagan and Selsor go on to classify the $maj$-Wilf equivalence classes for all subsets of $S_3$ and one can translate these results into equivalent statements for the charge statistic utilizing the function $f$ if desired.

\section{A Conjecture of Dokos, Dwyer, Johnson, Sagan and Selsor}

In their paper defining $st$-Wilf equivalence, Dokos, Dwyer, Johnson, Sagan and Selsor \cite{DDJ} state the following conjecture (Conjecture 3.6):

\begin{Conjecture}
For all $k \geq 0$ we have
\[
<q^i> M_{2^k-1}(321;q) = \begin{cases}
1 & {\text {if }} i = 0,\cr
{\text {an even number}} & {\text {if }} i \geq 1.
\end{cases}
\]
\end{Conjecture} 

We will prove the analogous statement for the charge statistic in this section.

Before proving this theorem, we need some preliminary results and definitions.  It is a well-known result of the Robinson-Schensted correspondence that if two permutations have the same $P$ tableau then the charge statistic on those two permutations is the same.  In addition, one can easily show that a permutation is $321$-avoiding if and only if the $P$ tableau contains at most 2 rows.  There is only one tableau with 1 row, namely $P = \begin{array}{ccccc} 1&2&3&\cdots&n\end{array}$ and this tableau corresponds with the permutation $\pi = \begin{array}{ccccc}1&2&3&\cdots&n\end{array}$ which has a charge value of zero.  

We will now inductively define a bijection $\phi_{2^k-1}$ from the set of 2-row tableau of size $2^k-1$ to the set of 2-row tableau of size $2^k-1$.  If $k$ = 1, there are no 2-row tableau of size $1$ so we begin with $k=2$.  There are two tableau of size $2^2-1=3$ and we define $\phi_3$ as:

\begin{align*}
\phi_3\left(\begin{array}{cc}1&2\\3& \end{array}\right) &= \begin{array}{cc}1&3\\2& \end{array}\\
\phi_3\left(\begin{array}{cc}1&3\\2& \end{array}\right) &= \begin{array}{cc}1&2\\3& \end{array}\end{align*}

Now suppose that $\phi_{2^{k-1}-1}$ is a bijection from the set of 2-row tableau of size $2^{k-1}-1$ to the set of 2-row tableau of size $2^{k-1}-1$.  Define $\phi_{2^k-1}$ as follows:

Let $T$ be a 2-row tableau of size $2^k-1$ and let $S$ be the portion of $T$ containing the numbers $1$, $2$, $\dots$, $2^{k-1}-1$.  

\vspace{.2in}
\underline{\bf{Case 1:}}  If $S$ is a 2-row tableau then define $\phi_{2^k-1}(T)$ as the tableau $T$ with $S$ replaced by $\phi_{2^{k-1}-1}(S)$.  
Since $\phi_{2^{k-1}-1}$ is a bijection then $\phi_{2^k-1}$ is a bijection on this set of tableau.

For example, 
\begin{Example}  Let $k=4$ and 
\[
T = \begin{array}{ccccccccc}1&2&4&5&7&8&9&13&15\\3&6&10&11&12&14& & & \end{array}
\]
Then
\[
S=\begin{array}{ccccc}1&2&4&5&7\\3&6& & & \end{array},
\]
and if
\[
\phi(S) = \begin{array}{ccccc} 1&3&4&5&7\\2&6& & & \end{array}
\]
then
\[
\phi_{2^4-1}(T) = \begin{array}{ccccccccc} 1&3&4&5&7&8&9&13&15\\2&6&10&11&12&14& & \end{array}.
\]
\end{Example}

\underline{\bf{Case 2:}}  If $S$ is a 1-row tableau, then let $l$ be the smallest number such that $2k$, $2k+1$, $\dots$, $2k+l$ are in one row of $T$ and $2k+l+1$, $2k+l+2$, $\dots$, $2k+2l$ are in the other row with $k \geq 0$ as small as possible.

Form $\phi_{2^k-1}(T)$ by swapping the positions of $2k$, $2k+1$, $\dots$, $2k+l$ and $2k+l+1$, $2k+l+2$, $\dots$, $2k+2l$ in $T$.  Clearly this is a bijection on this set of tableaux.  For example,
\begin{Example}
Let $k=4$ and let 
\[
T=\begin{array}{ccccccccccc} 1&2&3&4&5&6&7&8&9&11&14\\10&12&13&15& & & & & & & \end{array}.
\]
Then 
\[
\phi_{2^4-1}(T) = \begin{array}{ccccccccccc}1&2&3&4&5&6&7&8&9&10&14\\11&12&13&15& & & & & & & \end{array}.
\]

\end{Example}

Now we can move on to some necessary preliminary results.
\begin{Lemma}
For all $k \geq 1$, there are an odd number of 321-avoiding permutations in $S_{2^k-1}$.
\end{Lemma}

\begin{proof}
Every $321$-avoiding permutation corresponds to a pair of tableau with 1 or 2 rows.  There is only one pair of 1-row tableau, namely $(P, P)$ where $P = \begin{array}{ccccc} 1&2&3&\cdots&2^k-1 \end{array}$.  The number of pairs of 2-row tableau is $m^2$ where $m$ is the number of 2-row tableau of size $2^k-1$.  Since $\phi_{2^k-1}$ is a bijection on the set of 2-row tableau of size $2^k-1$ with $\phi_{2^k-1}(T) \neq T$ for any $T$ (by definition) and $\phi_{2^k-1}(\phi_{2^k-1}(T)) = T$, then $m$ is even and thus $m^2$ is even.
\end{proof}

\begin{Corollary} There are an even number of 2-row tableau of size $2^k-1$.
\end{Corollary}

\begin{Theorem}
For all $k \geq 0$ we have
\[
<q^i> Ch_{2^k-1}(321;q) = \begin{cases}
1 & {\text {if }} i = 0,\cr
{\text {an even number}} & {\text {if }} i \geq 1.
\end{cases}
\]
\end{Theorem} 

\begin{proof}  
Since any inversion in a permutation $\pi$ introduces a charge value, the only permutation in $S_{2^k-1}$ with a charge value of zero is $\pi = \begin{array}{ccccc}1&2&3&\cdots&n \end{array}$, which corresponds to the only 1-row tableau of size $n$.  Thus all other permutations in $Av(321)$ correspond to pairs of 2-row tableau under the Robinson-Schensted correspondence.  By Corollary, there are $m$ 2-row tableau of size $2^k-1$ and $m$ is even.  

Choose an ordered pair of distinct 2-row tableaux of size $2^k-1$, say $A$ and $B$.  This can be done in $m(m-1)$ ways so there are an even number of ordered pairs $(A,B)$ where $A \neq B$.  Then since any two permutations which give rise to the same $P$ tableau under the Robinson-Schensted correspondence have the same charge, the permutations corresponding to $(A,A)$ and $(A,B)$ have the same charge.  

Thus the set of pairs of 2-row tableau of size $2^k-1$ can be partitioned into pairs of pairs that give rise to a pair of permutations with the same charge.  This gives our result.
\end{proof}

We now obtain the conjecture of Dokos, Dwyer, Johnson, Sagan and Selsor as a Corollary.

\begin{Corollary}
For all $k \geq 0$ we have
\[
<q^i> M_{2^k-1}(321;q) = \begin{cases}
1 & {\text {if }} i = 0,\cr
{\text {an even number}} & {\text {if }} i \geq 1.
\end{cases}
\]
\end{Corollary} 

\begin{proof}
The function $f$ defined in Section 3 is a bijection from $Av(321)$ to $Av(321)$ that takes the major index to the charge statistic, thus applying $f$ to this set gives the result.
\end{proof}

\end{document}